\PassOptionsToPackage{dvipsnames}{xcolor} 
\documentclass[a4paper,11pt]{article} 
\usepackage{amssymb,amsmath,amsthm,multirow,color,xcolor,cancel,bbm,tikz,mathrsfs}

\tikzstyle{main node}=[draw,circle,inner sep=1,outer sep=2,thick,minimum size=12pt]

\newtheorem{lemma}{Lemma\setcounter{claimcounter}{0}}
\newtheorem{theorem}{Theorem\setcounter{claimcounter}{0}}
\newtheorem{corollary}{Corollary}

\newcounter{claimcounter}
\newtheorem*{claim*}{{\it Claim}}

\newcommand{\BF}[1]{{\boldmath{\bf #1}\unboldmath}}
\newcommand{\B}{\{0,1\}}
\newcommand{\ONE}{\mathbf{1}}
\newcommand{\ZERO}{\mathbf{0}}

\newcommand{\G}{\mathcal{G}}

\newcommand{\cst}{\mathrm{cst}}
\newcommand{\id}{\mathrm{id}}

\title{Isomorphic Boolean networks and dense interaction~graphs}

\author{Aymeric Picard Marchetto\footnote{Laboratoire I3S, UMR CNRS 7271 \& Universit\'e C\^ote d'Azur, France.}
~and Adrien Richard$^{*}$\\[3mm]

}

\begin{document}
%%%%%%%%%%%%%%%%%%%%%%%%%%%%%%%%%%%%%%%%%%%%%%%%%%%%%%%%%%%%%%%%%%%%%%%%%%%%%%%%%%%%%%%%%%%%%%%%%%%%%%%
%%%%%%%%%%%%%%%%%%%%%%%%%%%%%%%%%%%%%%%%%%%%%%%%%%%%%%%%%%%%%%%%%%%%%%%%%%%%%%%%%%%%%%%%%%%%%%%%%%%%%%%

\maketitle

\begin{abstract}
A Boolean network (BN) with $n$ components is a discrete dynamical system described by the successive iterations of a function $f:\B^n\to\B^n$. In most applications, the main parameter is the interaction graph of $f$: the digraph with vertex set $\{1,\dots,n\}$ that contains an arc from $j$ to $i$ if $f_i$ depends on input $j$. What can be said on the set $\G(f)$ of the interaction graphs of the BNs $h$ isomorphic to $f$, that is, such that $h\circ \pi=\pi\circ f$ for some permutation $\pi$ of $\B^n$? It seems that this simple question has never been studied. Here, we report some basic facts. First, if $n\geq 5$ and $f$ is neither the identity or constant, then $\G(f)$ is of size at least two and contains the complete digraph on $n$ vertices, with $n^2$ arcs. Second, for any $n\geq 1$, there are $n$-component BNs $f$ such that every digraph in $\G(f)$ has at least $n^2/9$ arcs.
\end{abstract}

\section{Introduction}

A \emph{Boolean network} (network for short) with $n$ components is a finite dynamical system defined by the successive iterations of a function  
\[
f:\B^n\to\B^n,\quad x=(x_1,\dots,x_n)\mapsto f(x)=(f_1(x),\dots,f_n(x)).
\]
Boolean networks have many applications; in particular, they are omnipresent in the modeling of neural and gene networks (see \cite{N15} for a review).

The ``network'' terminology comes from the fact that the \textit{interaction graph} $G(f)$ of $f$ is often considered as the main parameter of $f$: the vertex set is $[n]=\{1,\dots,n\}$ and there is an arc from $j$ to $i$ if $f_i$ depends on input $j$, that is, if there are $x,y\in\B^n$ which only differ in $x_j\neq y_j$ such that $f_i(x)\neq f_i(y)$. For instance, in the context of gene networks, the interaction graph is often well approximated while the actual dynamics is not. One is thus faced with the following question: {\em what can be said on the dynamics described by $f$ from $G(f)$ only}. There are many results in this direction (see \cite{G20} for a review). In most cases, the studied dynamical properties are invariant by isomorphism: number of fixed or periodic points, number and size of limit cycles, transient length and so on. However, the interaction graph is {\em not} invariant by isomorphism: even if $f$ and $h$ are isomorphic, their interaction graphs can be very different (by Theorem~1 below, $G(f)$ can have $n^2$ arcs while $G(h)$ has a single arc). Surprisingly, this variation seems to have never been studied, and we report here some basic~facts.  

Given a network $f$, let $\G(f)$ be the set of interaction graphs $G(h)$ such that $h$ is a network isomorphic to $f$, that is, such that $h\circ\pi=\pi\circ f$ for some permutation $\pi$ of $\B^n$. Hence, we propose to study $\G(f)$. For instance, if $f$ is constant, we write this $f=\cst$, then $\G(f)$ contains a single digraph, the digraph on $[n]$ without arcs, and if $f$ is the identity, we write this $f=\id$, then $\G(f)$ also contains a single digraph, the digraph on $[n]$ with $n$ loops (cycles of length one) and no other arcs. 

Our first result shows that, excepted few exceptions (including the two above examples), $\G(f)$ always contains the complete digraph on $[n]$, with $n^2$ arcs, denoted $K_n$. 

\begin{theorem}\label{thm:main1}
$K_n\in\G(f)$ for all network $f\neq\cst,\id$ with $n\geq 5$ components.
\end{theorem}

Our second result shows that, for $n\geq 5$, there are no networks $f$ such that $\G(f)$ only contains the complete digraph. 

\begin{theorem}\label{thm:main2}
$\G(f)\neq\{K_n\}$ for all network $f$ with $n\geq 5$ components.
\end{theorem}

From these theorems we deduce the following property, which might seem innocente but that doesn't seem to have a one-line proof: {\em 
If $f$ is a network with $n\geq 5$ components, then $|\G(f)|=1$ if and only if $f=\id$ or~$f=\cst$.}

Even if, for $n\geq 5$, $\G(f)$ cannot only contain the complete digraph, using  a well known isoperimetric inequality in hypercubes, we show that, at least, $\G(f)$ can only contain digraphs with many arcs.

\begin{theorem}\label{thm:main3}
For every $n\geq 1$, there is a network $f$ with $n$ components such that every digraph in $\G(f)$ has at least $n^2/9$ arcs.
\end{theorem}

Concerning short term perspectives, we checked by computer that the first theorem hold for $n=3$ (for $n=2$ it fails; see the appendix) and that the second holds for $n=2,3$. For $n=4$ this is much time-consuming and we didn't do it, hoping to find instead dedicated arguments, since those given for $n\geq 5$ do not work. A more involved perspective, related to the third theorem, is to prove that, for any $\epsilon>0$ and $n$ large enough, any digraph in $\G(f)$ has at least $(1-\epsilon)n^2$ arcs for some $f$. We also want to study networks $f$ such that $\G(f)$ is very large and ask: do we have, for any $\epsilon>0$ and $n$ large enough, $|\G(f)|/2^{n^2}>1-\epsilon$ for some $f$?   

The three above theorems are proved the in following three sections. Before going on, we give some basic definitions. An element $x$ of $\B^n$ is a {\em configuration}, and elements in $[n]$ are {\em components}. We set $\ONE(x)=\{i\in [n]\mid x_i=1\}$ and $\ZERO(x)=[n]\setminus \ONE(x)$. The {\em weight} of $x$ is $w(x)=|\ONE(x)|$. We denote by $e_i$ the configuration such that $\ONE(e_i)=\{i\}$. For $x,y\in\B^n$, the sum $x+y$ is applied component-wise modulo two. Hence $x$ and $x+e_i$ only differ in component $i$. We denote by $\ONE$ (resp. $\ZERO$) the configuration of weight $n$ (resp. $0$). Thus $x$ and $x+\ONE$ differ in every component. Given $A\subseteq \B^n$, we set $A+x=\{a+x\mid a\in A\}$. Let $f$ be a $n$-component network. A {\em fixed point} is a configuration $x$ such that $f(x)=x$. Let $\Gamma(f)$ be the digraph with vertex set $\B^n$ and an arc from $x$ to $f(x)$ for every $x\in\B^n$. A {\em limit cycle} of $f$ is a cycle of $\Gamma(f)$. Hence fixed points correspond to limit cycles of length one. An {\em independent set} of $f$ is an independent set of $\Gamma(f)$, equivalently, it is a set $A\subseteq \B^n$ such that $f(A)\cap A=\emptyset$.  
 
%%%%%%%%%%%%%%%%%%%%%%%%%%%%%%%%%%%%%%%%%%%%%%%%%%%%%%%%%%%%%%%%%%%%%%%%%%%%%%%%%%%%%%%%%%%%%%%%%%%%%%%
%%%%%%%%%%%%%%%%%%%%%%%%%%%%%%%%%%%%%%%%%%%%%%%%%%%%%%%%%%%%%%%%%%%%%%%%%%%%%%%%%%%%%%%%%%%%%%%%%%%%%%%
\section{Proof of Theorem \ref{thm:main1}}
%%%%%%%%%%%%%%%%%%%%%%%%%%%%%%%%%%%%%%%%%%%%%%%%%%%%%%%%%%%%%%%%%%%%%%%%%%%%%%%%%%%%%%%%%%%%%%%%%%%%%%%
%%%%%%%%%%%%%%%%%%%%%%%%%%%%%%%%%%%%%%%%%%%%%%%%%%%%%%%%%%%%%%%%%%%%%%%%%%%%%%%%%%%%%%%%%%%%%%%%%%%%%%%

We proceed by showing that, for $n\geq 5$ and $f\neq\cst,\id$, we have $K_n\in\G(f)$ if at least one of the following three conditions holds: $f$ has at least $2n$ fixed points; or $f$ has at least $n$ limit cycles of length $\geq 3$; or $f$ has an independent set of size $\geq 2n$. We then prove that, because $n\geq 5$, at least one of the three conditions holds, and Theorem~\ref{thm:main1} follows.   

\begin{lemma}\label{lem:fixed_points}
Let $f$ be a network with $n$ components, which is not the identity. If $f$ has at least $2n$ fixed points then $K_n\in\G(f)$.  
\end{lemma}

\begin{proof}
Since $f\neq \id$ we have $f(c)\neq c$ for some $c$, and since $f$ has at least $2n$ fixed points, it has $2n-1$ fixed points distinct from $f(c)$, say $a^0,a^1,\dots,a^n,b^3,\dots,b^n$. Let $\pi$ be any permutation of $\B^n$ such that $\pi(a^0)=\ZERO$, $\pi(a^i)=e_i$ for $1\leq i\leq n$, $\pi(b^i)=e_1+e_2+e_i$ for $3\leq i\leq n$, $\pi(c)=e_1+e_2$ and  $\pi(f(c))=e_1+e_2+\ONE$. Let $h=\pi\circ f\circ \pi^{-1}$. We will prove that $G(h)=K_n$.
For $i\in [n]$, we have $h(\ZERO)=\pi(f(a^0))=\pi(a^0)=\ZERO$ and $h(e_i)=\pi(f(a^i))=\pi(a^i)=e_i$, hence $h(\ZERO)$ and $h(e_i)$ differ in component $i$, and we deduce that $G(h)$ has an arc from $i$ to itself.  
It remains to prove that $G(h)$ has an arc from $i$ to $j$ for distinct $i,j\in [n]$. We have $h(e_1+e_2)=\pi(f(c))=e_1+e_2+\ONE$. Hence $h(e_2)=e_2$ and $h(e_1+e_2)$ differ in every component $j\neq 1$, and thus $G(h)$ has an arc from $1$ to every $j\neq 1$. We prove similarly that $G(h)$ has an arc from $2$ to every $j\neq 2$. For $3\leq i\leq n$, we have $h(e_1+e_2+e_i)=\pi(f(b^i))=\pi(b^i)=e_1+e_2+e_i$, so $h(e_1+e_2+e_i)$ and $h(e_1+e_2)$ differ in every component $j\neq i$, and we deduce that $G(h)$ has an arc from $i$ to every $j\neq i$.  
\end{proof}

\begin{lemma}\label{lem:limitA_cycles}
Let $f$ be a network with $n$ components. If $f$ has at least $n$ limit cycles of length $\geq 3$, then $K_n\in\G(f)$.  
\end{lemma}

\begin{proof}
Suppose that $f$ has $n$ limit cycles of length $\geq 3$; this implies $n\geq 4$. Let $a^1,\dots,a^n$ be configurations inside distinct limit cycles of $f$ of length $\geq 3$. For $i\in [n]$ let $b^i=f(a^i)$ and $c^i=f(b^i)$. Then $a^1,\dots,a^n$, $b^1,\dots,b^n$, $c^1,\dots,c^n$ are all distinct. For $i\in [n]$, let $x^i=e_{i-1}+e_i$,  $y^i=e_{i-1}$ and $z^i=e_{i-1}+\ONE$, where $e_0$ means $e_n$. Since $n\geq 4$, the configurations $x^1,\dots,x^n$, $y^1,\dots,y^n$, $z^1,\dots,z^n$ are all distinct. Hence there is a permutation $\pi$ of $\B^n$ such that, for $i\in [n]$, $\pi(a^i)=x^i$, $\pi(b^i)=y^i$ and $\pi(c^i)=z^i$. Let $h=\pi\circ f\circ\pi^{-1}$. For $i\in [n]$ we have $h(e_{i-1}+e_i)=h(x^i)=\pi(f(a^i))=\pi(b^i)=y^i=e_{i-1}$ and $h(e_{i-1})=h(y^i)=\pi(f(b^i))=\pi(c^i)=z^i=e_{i-1}+\ONE$. Hence $h(e_{i-1}+e_i)$ and $h(e_{i-1})$ differ in every component, and since $e_{i-1}+e_i$ and $e_{i-1}$ only differ in component $i$, we deduce that $G(h)$ has an arc from $i$ to every $j\in [n]$. Thus $G(h)=K_n$. 
\end{proof}

\begin{lemma}\label{lem:independent_set}
Let $f$ be a non-constant network with $n\geq 5$ components. If $f$ has an independent set of size $\geq 2n$, then $K_n\in\G(f)$.  
\end{lemma}

\begin{proof}
Let $A$ be an independent set of $f$. We first prove:
\begin{itemize}
\item[(1)]
{\em If $|A|\geq n+k$ and $|f(A)|=2k$ for some $1\leq k\leq n$, then $K_n\in\G(f)$.}
\end{itemize}
Suppose these condition holds. One easily check that there is an independent set $A$ with $|A|=n+k$ and $|f(A)|=2k$ for some $1\leq k\leq n$. Let us write $f(A)=\{a^1,\dots,a^{2k}\}$, and let $A_p=f^{-1}(a^p)\cap A$ for $p\in[2k]$. Let $X_1,\dots,X_{2k}$ be disjoint subsets of $\B^n$ of size $|A_1|,\dots,|A_{2k}|$ such that, for all $i\in [n]$, there is $p\in [k]$ and $x\in X_{2p-1}$ with $x+e_i\in X_{2p}$; that these sets exist is the ``technical'' part of the proof, given by Lemma~\ref{lem:X} in appendix.

Let $X=X_1\cup\dots\cup X_{2k}$, let $Y$ be the set of $y\in\B^n$ with $y_1=0$, and let $Y'$ be the set of $y\in Y$ such that $y,y+\ONE\not\in X$. Since $n\geq 5$ and $n\geq k$:
\[
|Y'|\geq |Y|-|X|=2^{n-1}-(n+k)\geq 2^{n-1}-2n\geq n\geq k.
\] 
Thus there are $k$ distinct configurations in $y^1,\dots,y^k\in Y'$ and by construction, $Y''=\{y^1,\dots,y^k\}$ and $Y''+\ONE$ are disjoint and disjoint from~$X$.  

Hence there is a permutation $\pi$ of $\B^n$ such that, for all $p\in [k]$, $\pi(a^{2p-1})=y^p$, $\pi(a^{2p})=y^p+\ONE$, $\pi(A_{2p-1})=X_{2p-1}$ and $\pi(A_{2p})=X_{2p}$. Let $h=\pi\circ f\circ\pi^{-1}$. By construction, for every $i\in [n]$ there is $p\in [k]$ and $x\in X_{2p-1}$ with $x+e_i\in X_{2p}$. Since $\pi^{-1}(x)\in A_{2p-1}$ and $\pi^{-1}(x+e_i)\in A_{2p}$, we have $h(x)\in\pi(f(A_{2p-1}))=\{\pi(a^{2p-1})\}=\{y^p\}$ and $h(x+e_i)\in\pi(f(A_{2p}))=\{\pi(a^{2p})\}=\{y^p+\ONE\}$. Thus $h(x)$ and $h(x+e_i)$ differ in every component, and we deduce that $G(h)=K_n$. This proves (1).

\medskip
We next prove another condition on $A$ to obtain the complete digraph. 
\begin{itemize}
\item[(2)]
{\em If $|A|>n$ and $|f(A)|=1$, then $K_n\in\G(f)$.}
\end{itemize}
Let $a\in\B^n$ such that $f(A)=\{a\}$. Since $f\neq\cst$, there is $b\in\B^n$ with $f(b)\neq a$, and thus $b\not\in A$.  We consider three cases. 

Suppose first that $f(a)\neq a$. Since $|A|>n$, there are $n$ configurations $a^1,\dots,a^n$ in $A$ distinct from $f(a)$. Then $a^1,\dots,a^n,a,f(a)$ are all distinct, so there is a permutation $\pi$ with $\pi(a)=\ZERO$, $\pi(f(a))=\ONE$ and $\pi(a^i)=e_i$ for $1\leq i\leq n$. Let $h=\pi\circ f\circ\pi^{-1}$. For $i\in [n]$, we have $h(e_i)=\pi(f(a^i))=\pi(a)=\ZERO$ and $h(\ZERO)=\pi(f(a))=\ONE$, so $h(e_i)$ and $h(\ZERO)$ differ in every component, and thus $G(h)=K_n$.  

Suppose now that $f(a)=a$ and $f(b)=b$. Let $a^1,\dots,a^n\in A$, all distinct. Then $a^1,\dots,a^n,a,b$ are all distinct since $b=f(b)\neq a$. So there is a permutation $\pi$ with $\pi(a)=\ONE$, $\pi(b)=\ZERO$ and $\pi(a^i)=e_i$ for $i\in [n]$. Let $h=\pi\circ f\circ\pi^{-1}$. For all $i\in [n]$, we have $h(e_i)=\pi(f(a^i))=\pi(a)=\ONE$ and $h(\ZERO)=\pi(f(b))=\pi(b)=\ZERO$, so $h(e_i)$ and $h(\ZERO)$ differ in every component, and thus $G(h)=K_n$.  

Suppose finally that $f(a)=a$ and $f(b)\neq b$. Since $|A|>n$, there is $A'\subseteq A\setminus\{f(b)\}$ of size $n$.  Then $A'\cup\{b\}$ is an independent set of size $n+1$ and $|f(A'\cup\{b\})|=|\{a,f(b)\}|=2$ so $K_n\in \G(f)$ by (1). This proves (2).  

\medskip
We can now conclude the proof. Suppose that $|A|\geq 2n$. Then we can choose $A$ so that $|A|=2n$. Suppose, for a contradiction, that $K_n\not\in\G(f)$. If $|f(A)|$ is even then  $K_n\in\G(f)$ by (1) and if $|f(A)|=1$ then  $K_n\in\G(f)$ by (2). Thus $|f(A)|=2k+1$ for some $1\leq k<n$. Let us write $f(A)=\{a^1,\dots,a^{2k+1}\}$, and let $A_p=f^{-1}(a^p)\cap A$ for $1\leq p\leq 2k+1$. Suppose, without loss, that $|A_1|\leq |A_2|\leq \dots\leq |A_{2k+1}|$. Then $A'=A\setminus A_1$ is an independent set with $|f(A')|=2k$. Setting $m=n+k-1$, if $|A'|>m$, then $K_n\in\G(f)$ by (2). So $2k|A_2|\leq |A'|\leq m$, thus $|A_1|\leq |A_2|\leq m/2k$. We deduce that $2n=|A|=|A_1|+|A'|\leq m/2k+m$. However, one easily checks that $2n> m/2k+m$ when $1\leq k<n$, a contradiction. Thus $K_n\in\G(f)$
\end{proof}

We are ready to prove Theorem~\ref{thm:main1}. Let $f\neq\cst,\id$ with $n\geq 5$ components. Suppose, for a contradiction, that $K_n\not\in\G(f)$. Let $F$ be the set of fixed points of $f$, and let $L$ be a smallest subset of $\B^n$ intersecting every limit cycle of $f$ of length $\geq 3$. Let $\Gamma'$ be obtained from $\Gamma(f)$ by deleting the vertices in $F\cup L$; then $\Gamma'$ has only cycles of length two, thus it is bipartite. Since $K_n\not\in\G(f)$, by Lemmas~\ref{lem:fixed_points} and \ref{lem:limitA_cycles}, we have $|F|<2n$ and $|L|<n$, thus $\Gamma'$ has at least $N=2^n-3n+2$ vertices. Since $\Gamma'$ is bipartite, it has an independent set $A$ of size $\geq N/2$. Then $A$ is an independent set of $f$ and since $K_n\not\in\G(f)$, we deduce from Lemma~\ref{lem:independent_set} that $|A|<2n$. Thus $2n>N/2$, that is, $7n>2^n+3$, which is false since $n\geq 5$. Thus $K_n\in\G(f)$.  

%%%%%%%%%%%%%%%%%%%%%%%%%%%%%%%%%%%%%%%%%%%%%%%%%%%%%%%%%%%%%%%%%%%%%%%%%%%%%%%%%%%%%%%%%%%%%%%%%%%%%%%
%%%%%%%%%%%%%%%%%%%%%%%%%%%%%%%%%%%%%%%%%%%%%%%%%%%%%%%%%%%%%%%%%%%%%%%%%%%%%%%%%%%%%%%%%%%%%%%%%%%%%%%
\section{Proof of Theorem \ref{thm:main2}}
%%%%%%%%%%%%%%%%%%%%%%%%%%%%%%%%%%%%%%%%%%%%%%%%%%%%%%%%%%%%%%%%%%%%%%%%%%%%%%%%%%%%%%%%%%%%%%%%%%%%%%%
%%%%%%%%%%%%%%%%%%%%%%%%%%%%%%%%%%%%%%%%%%%%%%%%%%%%%%%%%%%%%%%%%%%%%%%%%%%%%%%%%%%%%%%%%%%%%%%%%%%%%%%

We first give a necessary and sufficient condition for the presence of a digraph in $\G(f)$ that misses some arc with distinct ends. It has been obtained with K\'evin Perrot, whom we thank, and its proof is in appendix. Given an $n$-component network $f$, and $1\leq k\leq 2^{n-1}$, a {\em $k$-nice set} of $f$ is a set $A\subseteq \B^n$ of size $2k$ with $|f^{-1}(A)|$ and $|f^{-1}(A)\cap A|$ even.

\begin{lemma}\label{lem:nice}
Let $f$ be a $n$-component network and distinct $i,j\in [n]$. Some digraph in $\G(f)$ has no arc from $j$ to $i$ if and only if $f$ has a $(2^{n-2})$-nice~set. 
\end{lemma}

Hence, to prove Theorem \ref{thm:main2}, it is sufficient to prove that, for $n\geq 5$, $f$ has always a $(2^{n-2})$-nice set. To use induction, it is more convenient to prove something stronger: if $n\geq 4$ and $8\leq k\leq 2^{n-1}$ then $f$ has always a $k$-nice set. In particular, if $n\geq 5$ then $2^{n-2}\geq 8$ thus $f$ has a $(2^{n-2})$-nice set and we are done. So it remains to prove:

\begin{lemma}
Let $f$ be a network with $n\geq 4$ components and $8\leq k\leq 2^{n-1}$. Then $f$ has a $k$-nice set.
\end{lemma}

\begin{proof}
We proceed by induction on $k$, decreasing from $2^{n-1}$ to $8$. For the base case, observe that $\B^n$ is a $(2^{n-1})$-nice set of $f$. Suppose that $8<k\leq 2^{n-1}$, and suppose that $f$ has a $k$-nice set $A$. We will prove that $f$ has a $(k-1)$-nice set included in $A$, thus completing the induction step. For that, we will define an equivalence relation on $A$, with at most $8$ equivalence classes, and show that, since $A\geq 18$, there there always are two equivalent elements $x,y\in A$ such that $A\setminus\{x,y\}$ is a $(k-1)$-nice set of $f$. This relation is defined through $3$ binary properties on the elements of $A$, defined below. 

Let $\alpha,\beta,\gamma:A\to\B$ be defined by: for all $x\in A$, 
\begin{enumerate}
\item $\alpha(x)=1$ if and only if $f(x)\in A$,
\item $\beta(x)=1$ if and only if $|f^{-1}(x)|$ is even, 
\item $\gamma(x)=1$ if and only if $|(f^{-1}(x)\cap A)\setminus\{x\}|$ is even. 
\end{enumerate}
We say that $x,y\in A$ are {\em equivalent} if $\alpha(x)=\alpha(y)$ and $\beta(x)=\beta(y)$ and $\gamma(x)=\gamma(y)$. We say that $x,y\in A$ are {\em independent} if $f(x)\neq y$ and $f(y)\neq x$. It is straightforward (and annoying) to check that if $x,y\in A$ are equivalent and independent, then $A\setminus\{x,y\}$ is a $(k-1)$-nice of $f$. 

Suppose now that $f$ has no $(k-1)$-nice set. By the above property, there is no equivalence class containing two independent elements. Let $A_0=\alpha^{-1}(0)$ and $A_1=\alpha^{-1}(1)$. Each of $A_0,A_1$ contains at most $4$ equivalence classes. If $x,y\in A_0$, then $f(x),f(y)\not\in A$, and thus $x$ and $y$ are independent. We deduce that each of the classes contained in $A_0$ is of size $<2$, and thus $|A_0|\leq 4$. Also, one easily checks that any class of size $\geq 4$ contains two independent elements. Hence each of the classes contained in $A_1$ is of size $<4$, and thus $|A_1|\leq 12$. But then $|A|=|A_1|+|A_2|\leq 16$, a contradiction. Thus $f$ has a $(k-1)$-nice set.
\end{proof}

%Let $D$ be the digraph with vertex set $X$ and an arc from $x$ to $y$ if $x\neq y$ and $f(x)=y$. We will prove that $D$ has two independent vertices. This is obvious if $D$ is acyclic or has a cycle of length $\geq 4$. Otherwise, $D$ has a cycle $C$ of length $2$ or $3$. Since $|X|\geq 4$ there is a vertex $x$ not in $C$ and since $C$ is of length $\geq 2$, there is a vertex $y$ in $C$ distinct from $f(x)$. Then $x$ and $y$ are independent. Hence $X$ contains two independent elements, and thus, by ($*$), there is a $(k-1)$-nice set, a contradiction.

%%%%%%%%%%%%%%%%%%%%%%%%%%%%%%%%%%%%%%%%%%%%%%%%%%%%%%%%%%%%%%%%%%%%%%%%%%%%%%%%%%%%%%%%%%%%%%%%%%%%%%%
%%%%%%%%%%%%%%%%%%%%%%%%%%%%%%%%%%%%%%%%%%%%%%%%%%%%%%%%%%%%%%%%%%%%%%%%%%%%%%%%%%%%%%%%%%%%%%%%%%%%%%%
\section{Proof of Theorem \ref{thm:main3}}
%%%%%%%%%%%%%%%%%%%%%%%%%%%%%%%%%%%%%%%%%%%%%%%%%%%%%%%%%%%%%%%%%%%%%%%%%%%%%%%%%%%%%%%%%%%%%%%%%%%%%%%
%%%%%%%%%%%%%%%%%%%%%%%%%%%%%%%%%%%%%%%%%%%%%%%%%%%%%%%%%%%%%%%%%%%%%%%%%%%%%%%%%%%%%%%%%%%%%%%%%%%%%%%

If $1\leq n\leq 9$ and $f$ is the identity on $\B^n$, then $\G(f)$ contains a unique digraph, which has $n=n^2/n\geq n^2/9$ vertices.  This proves the theorem for $n\leq 9$. We treat the other cases with an explicit construction. 

Given $A\subseteq \B^n$ with $\ZERO\in A$, we denote by $f^A$ the $n$-component network such that $f^A(A)=\{\ZERO\}$ and $f(b)=b$ for all $b\not\in A$. We will prove, in the next lemma, that if the size of $A$ is carefully chosen, then, independently on the structure of $A$, $G(f^A)$ has at least $n^2/9$ arcs. It is then easy to prove that this lower bound holds for every digraph in $\G(f^A)$.

\begin{lemma}\label{lem:f^A}
Let $n\geq 9$ and $A\subseteq \B^n$ of size $\lceil 2^{n/4}\rceil$ with $\ZERO\in A$. Then $G(f^A)$ has at least $n^2/9$ arcs. 
\end{lemma} 

The key tool is the following lemma, which can be easily deduced from the well known Harper's isoperimetric inequality in the hypercube (see the appendix). Let $Q_n$ be the hypercube graph, with vertex set $\B^n$ and an edge between $x$ and $y$ if and only if $x$ and $y$ differ in exactly one component. 

\begin{lemma}\label{lem:cube}
Let $A\subseteq \B^n$ be non-empty, and let $d$ be the average degree of the subgraph of $Q_n$ induced by $A$. Then $|A|\geq 2^d$.
\end{lemma}

\begin{proof}[\BF{Proof of Lemma~\ref{lem:f^A}}]
Let $f=f^A$ and let $m$ be the number of arcs in $G(f)$. Suppose, for a contradiction that $m<n^2/9$. Let $I$ be the set of $i\in [n]$ with $a_i=1$ for some $a\in A$. Note that $|I|\geq n/4$. Indeed, let $B$ be the set of $b\in\B^n$ such that $b_i=0$ for all $i\not\in I$. Then $A\subseteq B$, and $|B|=2^{|I|}$, thus $|I|= \log_2 |B|\geq \log_2|A|\geq n/4$. From that observation, we can say something stronger:

\begin{itemize}
\item[(1)] {\em $|I|>n/2$.}
\end{itemize}

\noindent
Let $i\in I$ and $j\not\in I$. Let $a\in A$ such that $a_i=1$. Since $j\not\in I$, $a+e_j\not\in A$, thus $f(a)=\ZERO$ and $f(a+e_j)=a+e_i$, and $f_i(a+e_j)=a_i=1$ since $j\neq i$. Hence $G(f)$ has an arc from $j$ to $i$. We deduce that $|I|\cdot (n-|I|)\leq m<n^2/9$, so $n/4\cdot (n-|I|)<n^2/9$. Thus $|I|>5n/9>n/2$. This proves (1). 

\medskip
We deduce that some $a\in A$ has a large weight.

\begin{itemize}
\item[(2)] {\em There is $a\in A$ with $w(a)>3n/4$.}
\end{itemize}

\noindent
Suppose that $w(a)\leq 3n/4$ for all $a\in A$. Let $i\in I$. Let $a\in A$ with $a_i=1$, and $w(a)$ maximal for that property. Then, for all $j\in\ZERO(a)$, we have $a+e_j\not\in A$. Hence $f(a)=\ZERO$ and $f(a+e_j)=a+e_j$ so $f_i(a+e_j)=a_i=1$ since $i\neq j$. Thus $G(f)$ has an arc from $j$ to $i$. We deduce that the in-degree of $i$ in $G(f)$ is at least $|\ZERO(a)|=n-w(a)\geq n/4$. Hence $m\geq |I|\cdot n/4$ and since $|I|>n/2$ by (1) we obtain $m>n^2/8>n^2/9$, a contradiction. This proves (2).

\medskip
Let $A'$ be the set of $a\in A$ with $w(a)\geq n/3-1$, which is not empty by (2). For $a\in A$, let $J(a)$ be the set of $j\in [n]$ such that $a+e_j\in A$, and let $J'(a)$ be the set of $j\in [n]$ such that $a+e_j\in A'$. Let $H$ be the subgraphs of $Q_n$ induced by $A'$. If $a\in A'$, then $d(a)=|J'(a)|$ is the degree of $a$ in $H$. We will prove in (3) and (4) below that each vertex in $H$ has  large degree. 

\begin{itemize}
\item[(3)] {\em If $a\in A'$ and $w(a)\geq n/3$ then $d(a)>n/3$.}
\end{itemize}

\noindent
Indeed, since $w(a)\geq n/3$, for $j\in J(a)$ we have $w(a+e_j)\geq n/3-1$ thus $j\in J'(a)$. Hence $J(a)=J'(a)$ so $d(a)=|J(a)|$. For all $i\not\in J(a)$ we have $f(a)=\ZERO$ and $f(a+e_i)=a+e_i$, thus $G(f)$ has an arc from $i$ to each component in $\ONE(a+e_i)$. Thus the out-degree of $i$ in $G(f)$ has at least $w(a+e_i)\geq n/3-1$. We deduce that $(n-|J(a)|)\cdot (n/3-1)\leq m<n^2/9$, so $d(a)=|J(a)|>2n/3-3$ and since $n\geq 9$ we obtain $d(a)>n/3$. This proves (3). 

\begin{itemize}
\item[(4)] {\em If $a\in A'$ and $w(a)<n/3$ then $d(a)>n/3$.}
\end{itemize}

\noindent
Let $K=\ZERO(a)\setminus J'(a)$ and $i\in K$. We have $w(a+e_i)=w(a)+1\geq n/3$, thus if $a+e_i\in A$ then $a+e_i\in A'$; but then $i\in J'(a)$, a contradiction. Thus $a+e_i\not\in A$, so $f(a)=\ZERO$ and $f(a+e_i)=a+e_i$. We deduce that the out-degree of $i$ in $G(f)$ is at least $w(a+e_i)\geq n/3$. Thus $|K|\cdot n/3\leq m<n^2/9$ so $|K|<n/3$. Since $w(a)<n/3$ we have $|\ZERO(a)|>2n/3$. So $|K|\geq 2n/3-|J'(a)|$ and thus $d(a)=|J'(a)|>n/3$. This proves (4).

\medskip
We are now in position to finish the proof. Let $d$ be the average degree of $H$. By $(3)$ and $(4)$, we have $d>n/3$. Using Lemma~\ref{lem:cube} we obtain $2^{n/3}<2^d\leq |A'|\leq |A|=\lceil 2^{n/4}\rceil$, which is false because $n\geq 9$. 
\end{proof}

We are now in position to conclude. Let $n\geq 9$ and $A\subseteq \B^n$ of size $\lceil 2^{n/4}\rceil$ with $\ZERO\in A$, and let $f=f^A$. Let $\pi$ be any permutation of $\B^n$ and $h=\pi\circ f\circ\pi^{-1}$. We will prove that $G(h)$ has at least $n^2/9$ arcs. Let $A'=\pi(A)$ and $z=\pi(\ZERO)\in A'$. We have $h(A')=\pi(f(A))=\pi(\ZERO)=z$, and for $b\not\in A'$, $\pi^{-1}(b)\not\in A$ thus $h(b)=\pi(f(\pi^{-1}(b)))=\pi(\pi^{-1}(b))=b$. Let $h'$ be the $n$-component network defined by $h'(x)=h(x+z)+z$ for all $x\in\B^n$. It is easy to check that $G(h')=G(h)$. Let $A''=A'+z$, and note that $A''+z=A'$ and $|A''|=|A'|=|A|$, and also $\ZERO\in A''$ since $z\in A'$. We have $h'(A'')=h(A''+z)+z=h(A')+z=z+z=\ZERO$. Furthermore, for $b\not\in A''$, we have $b+z\not\in A'$ thus $h'(b)=h(b+z)+z=b+z+z=b$. Hence $h'=f^{A''}$ and since $|A''|=|A|=\lceil 2^{n/4}\rceil$, by Lemma~\ref{lem:f^A}, $G(h')$ has at least $n^2/9$ arcs.

%%%%%%%%%%%%%%%%%%%%%%%%%
\bibliographystyle{plain}
\bibliography{BIB}
%%%%%%%%%%%%%%%%%%%%%%%%%

\appendix

\section{Examples}

Consider the $6$ following networks with two components, denoted from $f^1$ to $f^6$ (each is given under three forms: a table, a graph, and~formulas). 
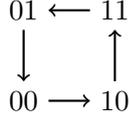
\[
\begin{array}{l}
%f1
\begin{array}{ccccc}
\begin{array}{c|c}
x & f^1(x)\\\hline
00&10\\
01&00\\
10&11\\
11&01
\end{array}
&&
\begin{array}{c}
\begin{tikzpicture}
\useasboundingbox (-0.3,-0.6) rectangle (1.5,1.8);
\node (00) at (0,0){$00$};
\node (01) at (0,1.2){$01$};
\node (10) at (1.2,0){$10$};
\node (11) at (1.2,1.2){$11$};
\path[thick,->,draw,black]
(00) edge (10)
(01) edge (00)
(10) edge (11)
(11) edge (01)
;
\end{tikzpicture}
\end{array}
&&
\begin{array}{l}
f^1_1(x)=x_2+1\\[2mm]
f^1_2(x)=x_1\\
\end{array}
\end{array}
\\~\\
%f2
\begin{array}{ccccc}
\begin{array}{c|c}
x & f^2(x)\\\hline
00&10\\
01&11\\
10&00\\
11&01
\end{array}
&&
\begin{array}{c}
\begin{tikzpicture}
\useasboundingbox (-0.3,-0.6) rectangle (1.5,1.8);
\node (00) at (0,0){$00$};
\node (01) at (0,1.2){$01$};
\node (10) at (1.2,0){$10$};
\node (11) at (1.2,1.2){$11$};
\path[thick,->,draw,black]
(00) edge[bend left=10] (10)
(10) edge[bend left=10] (00)
(01) edge[bend left=10] (11)
(11) edge[bend left=10] (01)
;
\end{tikzpicture}
\end{array}
&&
\begin{array}{l}
f^2_1(x)=x_1+1\\[2mm]
f^2_2(x)=x_2\\
\end{array}
\end{array}
\\~\\
%f3
\begin{array}{ccccc}
\begin{array}{c|c}
x & f^3(x)\\\hline
00&00\\
01&10\\
10&11\\
11&01
\end{array}
&&
\begin{array}{c}
\begin{tikzpicture}
\useasboundingbox (-0.3,-0.6) rectangle (1.5,1.8);
\node (00) at (0,0){$00$};
\node (01) at (0,1.2){$01$};
\node (10) at (1.2,0){$10$};
\node (11) at (1.2,1.2){$11$};
\draw[->,thick] (00.-112) .. controls (-0.5,-0.7) and (0.5,-0.7) .. (00.-68);
\path[thick,->,draw,black]
(01) edge (10)
(10) edge (11)
(11) edge (01)
;
\end{tikzpicture}
\end{array}
&&
\begin{array}{l}
f^3_1(x)=x_1+x_2\\[2mm]
f^3_2(x)=x_1\\
\end{array}
\end{array}
\\~\\
%f4
\begin{array}{ccccc}
\begin{array}{c|c}
x & f^4(x)\\\hline
00&00\\
01&11\\
10&10\\
11&01
\end{array}
&&
\begin{array}{c}
%\Gamma(f^3)\\[2mm]
\begin{tikzpicture}
\useasboundingbox (-0.3,-0.6) rectangle (1.5,1.8);
\node (00) at (0,0){$00$};
\node (01) at (0,1.2){$01$};
\node (10) at (1.2,0){$10$};
\node (11) at (1.2,1.2){$11$};
\draw[->,thick] (00.-112) .. controls (-0.5,-0.7) and (0.5,-0.7) .. (00.-68);
\draw[->,thick] (10.-112) .. controls ({1.2-0.5},-0.7) and ({1.2+0.5},-0.7) .. (10.-68);
\path[thick,->,draw,black]
(01) edge[bend left=10] (11)
(11) edge[bend left=10] (01)
;
\end{tikzpicture}
\end{array}
&&
\begin{array}{l}
f^4_1(x)=x_1+x_2\\[2mm]
f^4_2(x)=x_2\\
\end{array}
\end{array}
\\~\\
%f5
\begin{array}{ccccc}
\begin{array}{c|c}
x & f^5(x)\\\hline
00&01\\
01&11\\
10&11\\
11&01
\end{array}
&&
\begin{array}{c}
\begin{tikzpicture}
\useasboundingbox (-0.3,-0.6) rectangle (1.5,1.8);
\node (00) at (0,0){$00$};
\node (01) at (0,1.2){$01$};
\node (10) at (1.2,0){$10$};
\node (11) at (1.2,1.2){$11$};
\path[thick,->,draw,black]
(01) edge[bend left=10] (11)
(11) edge[bend left=10] (01)
(00) edge (01)
(10) edge (11)
;
\end{tikzpicture}
\end{array}
&&
\begin{array}{l}
f^5_1(x)=x_1+x_2\\[2mm]
f^5_2(x)=1\\
\end{array}
\end{array}
\\~\\
%f6
\begin{array}{ccccc}
\begin{array}{c|c}
x & f^6(x)\\\hline
00&00\\
01&00\\
10&10\\
11&10
\end{array}
&&
\begin{array}{c}
\begin{tikzpicture}
\useasboundingbox (-0.3,-0.6) rectangle (1.5,1.8);
\node (00) at (0,0){$00$};
\node (01) at (0,1.2){$01$};
\node (10) at (1.2,0){$10$};
\node (11) at (1.2,1.2){$11$};
\draw[->,thick] (00.-112) .. controls (-0.5,-0.7) and (0.5,-0.7) .. (00.-68);
\draw[->,thick] (10.-112) .. controls ({1.2-0.5},-0.7) and ({1.2+0.5},-0.7) .. (10.-68);
\path[thick,->,draw,black]
(01) edge (00)
(11) edge (10)
;
\end{tikzpicture}
\end{array}
&&
\begin{array}{l}
f^6_1(x)=x_1\\[2mm]
f^6_2(x)=0\\
\end{array}
\end{array}
\end{array}
\]
One can check that, given a network $f\neq\cst,\id$ with $2$ components: {\em $K_2\not\in\G(f)$ if and only if $f$ is isomorphic to one of the network given above}. Thus these $6$ networks are the counter examples of Theorem~\ref{thm:main1} for $n=2$. For instance, there are $6$ networks isomorphic to $f^1$, denoted from $h^1$ to $h^6$ (with $h^1=f^1$, given below with their interactions graphs. We deduce that $\G(f^1)$ contains three digraphs, each distinct from $K_2$.  

\[
\begin{array}{l}
%h1
\begin{array}{ccccccc}
\begin{array}{c|c}
x & h^1(x)\\\hline
00&10\\
01&00\\
10&11\\
11&01
\end{array}
&&
\begin{array}{c}
\begin{tikzpicture}
\useasboundingbox (-0.3,-0.6) rectangle (1.5,1.8);
\node (00) at (0,0){$00$};
\node (01) at (0,1.2){$01$};
\node (10) at (1.2,0){$10$};
\node (11) at (1.2,1.2){$11$};
\path[thick,->,draw,black]
(00) edge (10)
(01) edge (00)
(10) edge (11)
(11) edge (01)
;
\end{tikzpicture}
\end{array}
&&
\begin{array}{l}
h^1_1(x)=x_2+1\textcolor{white}{+x_1}\\[2mm]
h^1_2(x)=x_1\textcolor{white}{+x_2+1}\\
\end{array}
&&
\begin{array}{c}
G(h^1)\\[2mm]
\begin{tikzpicture}
\useasboundingbox (-0.7,-0.4) rectangle (2.2,0.4);
\node[outer sep=1,inner sep=2,circle,draw,thick] (1) at (0,0){$1$};
\node[outer sep=1,inner sep=2,circle,draw,thick] (2) at (1.5,0){$2$};
\path[->,thick]
(2) edge[bend right=25] (1)
(1) edge[bend right=25] (2)
;
\end{tikzpicture}
\end{array}
\end{array}
\\~\\
%h2
\begin{array}{ccccccc}
\begin{array}{c|c}
x & h^2(x)\\\hline
00&01\\
01&11\\
10&00\\
11&01
\end{array}
&&
\begin{array}{c}
\begin{tikzpicture}
\useasboundingbox (-0.3,-0.6) rectangle (1.5,1.8);
\node (00) at (0,0){$00$};
\node (01) at (0,1.2){$01$};
\node (10) at (1.2,0){$10$};
\node (11) at (1.2,1.2){$11$};
\path[thick,->,draw,black]
(00) edge (01)
(01) edge (11)
(10) edge (00)
(11) edge (10)
;
\end{tikzpicture}
\end{array}
&&
\begin{array}{l}
h^2_1(x)=x_2\textcolor{white}{+x_1+1}\\[2mm]
h^2_2(x)=x_1+1\textcolor{white}{+x_2}\\
\end{array}
&&
\begin{array}{c}
G(h^2)\\[2mm]
\begin{tikzpicture}
\useasboundingbox (-0.7,-0.4) rectangle (2.2,0.4);
\node[outer sep=1,inner sep=2,circle,draw,thick] (1) at (0,0){$1$};
\node[outer sep=1,inner sep=2,circle,draw,thick] (2) at (1.5,0){$2$};
\path[->,thick]
(2) edge[bend right=25] (1)
(1) edge[bend right=25] (2)
;
\end{tikzpicture}
\end{array}
\end{array}
\\~\\
%h3
\begin{array}{ccccccc}
\begin{array}{c|c}
x & h^3(x)\\\hline
00&10\\
01&11\\
10&01\\
11&00
\end{array}
&&
\begin{array}{c}
\begin{tikzpicture}
\useasboundingbox (-0.3,-0.6) rectangle (1.5,1.8);
\node (00) at (0,0){$00$};
\node (01) at (0,1.2){$01$};
\node (10) at (1.2,0){$10$};
\node (11) at (1.2,1.2){$11$};
\path[thick,->,draw,black]
(00) edge (10)
(01) edge (11)
(10) edge (01)
(11) edge (00)
;
\end{tikzpicture}
\end{array}
&&
\begin{array}{l}
h^3_1(x)=x_1+1\textcolor{white}{+x_2}\\[2mm]
h^3_2(x)=x_1+x_2\textcolor{white}{+1}\\
\end{array}
&&
\begin{array}{c}
G(h^3)\\[2mm]
\begin{tikzpicture}
\useasboundingbox (-0.7,-0.4) rectangle (2.2,0.4);
\node[outer sep=1,inner sep=2,circle,draw,thick] (1) at (0,0){$1$};
\node[outer sep=1,inner sep=2,circle,draw,thick] (2) at (1.5,0){$2$};
\draw[->,thick] (1.{180+20}) .. controls (-0.8,-0.7) and (-0.8,+0.7) .. (1.{180-20});
\draw[->,thick] (2.{0+20}) .. controls ({1.5+0.8},+0.7) and ({1.5+0.8},-0.7) .. (2.{0-20});
\path[->,thick]
(1) edge (2)
;
\end{tikzpicture}
\end{array}
\end{array}
\\~\\
%h4
\begin{array}{ccccccc}
\begin{array}{c|c}
x & h^4(x)\\\hline
00&11\\
01&10\\
10&00\\
11&01
\end{array}
&&
\begin{array}{c}
\begin{tikzpicture}
\useasboundingbox (-0.3,-0.6) rectangle (1.5,1.8);
\node (00) at (0,0){$00$};
\node (01) at (0,1.2){$01$};
\node (10) at (1.2,0){$10$};
\node (11) at (1.2,1.2){$11$};
\path[thick,->,draw,black]
(00) edge (11)
(01) edge (10)
(10) edge (00)
(11) edge (01)
;
\end{tikzpicture}
\end{array}
&&
\begin{array}{l}
h^4_1(x)=x_1+1\textcolor{white}{+x_2}\\[2mm]
h^4_2(x)=x_1+x_2+1\\
\end{array}
&&
\begin{array}{c}
G(h^4)\\[2mm]
\begin{tikzpicture}
\useasboundingbox (-0.7,-0.4) rectangle (2.2,0.4);
\node[outer sep=1,inner sep=2,circle,draw,thick] (1) at (0,0){$1$};
\node[outer sep=1,inner sep=2,circle,draw,thick] (2) at (1.5,0){$2$};
\draw[->,thick] (1.{180+20}) .. controls (-0.8,-0.7) and (-0.8,+0.7) .. (1.{180-20});
\draw[->,thick] (2.{0+20}) .. controls ({1.5+0.8},+0.7) and ({1.5+0.8},-0.7) .. (2.{0-20});
\path[->,thick]
(1) edge (2)
;
\end{tikzpicture}
\end{array}
\end{array}
\\~\\
%h5
\begin{array}{ccccccc}
\begin{array}{c|c}
x & h^5(x)\\\hline
00&01\\
01&10\\
10&11\\
11&00
\end{array}
&&
\begin{array}{c}
\begin{tikzpicture}
\useasboundingbox (-0.3,-0.6) rectangle (1.5,1.8);
\node (00) at (0,0){$00$};
\node (01) at (0,1.2){$01$};
\node (10) at (1.2,0){$10$};
\node (11) at (1.2,1.2){$11$};
\path[thick,->,draw,black]
(00) edge (01)
(01) edge (10)
(10) edge (11)
(11) edge (00)
;
\end{tikzpicture}
\end{array}
&&
\begin{array}{l}
h^5_1(x)=x_1+x_2\textcolor{white}{+1}\\[2mm]
h^5_2(x)=x_2+1\textcolor{white}{+x_1}\\
\end{array}
&&
\begin{array}{c}
G(h^5)\\[2mm]
\begin{tikzpicture}
\useasboundingbox (-0.7,-0.4) rectangle (2.2,0.4);
\node[outer sep=1,inner sep=2,circle,draw,thick] (1) at (0,0){$1$};
\node[outer sep=1,inner sep=2,circle,draw,thick] (2) at (1.5,0){$2$};
\draw[->,thick] (1.{180+20}) .. controls (-0.8,-0.7) and (-0.8,+0.7) .. (1.{180-20});
\draw[->,thick] (2.{0+20}) .. controls ({1.5+0.8},+0.7) and ({1.5+0.8},-0.7) .. (2.{0-20});
\path[->,thick]
(2) edge (1)
;
\end{tikzpicture}
\end{array}
\end{array}
\\~\\
%h6
\begin{array}{ccccccc}
\begin{array}{c|c}
x & h^6(x)\\\hline
00&11\\
01&00\\
10&01\\
11&10
\end{array}
&&
\begin{array}{c}
\begin{tikzpicture}
\useasboundingbox (-0.3,-0.6) rectangle (1.5,1.8);
\node (00) at (0,0){$00$};
\node (01) at (0,1.2){$01$};
\node (10) at (1.2,0){$10$};
\node (11) at (1.2,1.2){$11$};
\path[thick,->,draw,black]
(00) edge (11)
(01) edge (00)
(10) edge (01)
(11) edge (10)
;
\end{tikzpicture}
\end{array}
&&
\begin{array}{l}
h^6_1(x)=x_1+x_2+1\\[2mm]
h^6_2(x)=x_2+1\textcolor{white}{+x_1}\\
\end{array}
&&
\begin{array}{c}
G(h^6)\\[2mm]
\begin{tikzpicture}
\useasboundingbox (-0.7,-0.4) rectangle (2.2,0.4);
\node[outer sep=1,inner sep=2,circle,draw,thick] (1) at (0,0){$1$};
\node[outer sep=1,inner sep=2,circle,draw,thick] (2) at (1.5,0){$2$};
\draw[->,thick] (1.{180+20}) .. controls (-0.8,-0.7) and (-0.8,+0.7) .. (1.{180-20});
\draw[->,thick] (2.{0+20}) .. controls ({1.5+0.8},+0.7) and ({1.5+0.8},-0.7) .. (2.{0-20});
\path[->,thick]
(2) edge (1)
;
\end{tikzpicture}
\end{array}
\end{array}
\end{array}
\]

A corollary of Theorems~\ref{thm:main1} and~\ref{thm:main2}, stated in the introduction, is that, for $n\geq 5$, we have $|\G(f)|=1$ if and only if $f=\id$ or $f=\cst$. This is also true for $n=3$ (checked by computer) and open for $n=4$. But for $n=2$, $f^2$ is a counter example, the only one with two components. Indeed, there are $3$ networks isomorphic to $f^2$, denoted from $g^1$ to $g^3$ (with $g^1=f^2$), given below, and they have the same interaction graph, with a loop on each vertex, as for the identity. 
\[
\begin{array}{l}
%g1
\begin{array}{ccccccc}
\begin{array}{c|c}
x & g^1(x)\\\hline
00&10\\
01&11\\
10&00\\
11&01
\end{array}
&&
\begin{array}{c}
\begin{tikzpicture}
\useasboundingbox (-0.3,-0.6) rectangle (1.5,1.8);
\node (00) at (0,0){$00$};
\node (01) at (0,1.2){$01$};
\node (10) at (1.2,0){$10$};
\node (11) at (1.2,1.2){$11$};
\path[thick,->,draw,black]
(00) edge[bend left=10] (10)
(01) edge[bend left=10] (11)
(10) edge[bend left=10] (00)
(11) edge[bend left=10] (01)
;
\end{tikzpicture}
\end{array}
&&
\begin{array}{l}
g^1_1(x)=x_1+1\\[2mm]
g^1_2(x)=x_2\\
\end{array}
&&
\begin{array}{c}
G(g^1)\\[2mm]
\begin{tikzpicture}
\useasboundingbox (-0.7,-0.4) rectangle (2.2,0.4);
\node[outer sep=1,inner sep=2,circle,draw,thick] (1) at (0,0){$1$};
\node[outer sep=1,inner sep=2,circle,draw,thick] (2) at (1.5,0){$2$};
\draw[->,thick] (1.{180+20}) .. controls (-0.8,-0.7) and (-0.8,+0.7) .. (1.{180-20});
\draw[->,thick] (2.{0+20}) .. controls ({1.5+0.8},+0.7) and ({1.5+0.8},-0.7) .. (2.{0-20});
\end{tikzpicture}
\end{array}
\end{array}
\\~\\
%g2
\begin{array}{ccccccc}
\begin{array}{c|c}
x & g^2(x)\\\hline
00&01\\
01&00\\
10&11\\
11&10
\end{array}
&&
\begin{array}{c}
\begin{tikzpicture}
\useasboundingbox (-0.3,-0.6) rectangle (1.5,1.8);
\node (00) at (0,0){$00$};
\node (01) at (0,1.2){$01$};
\node (10) at (1.2,0){$10$};
\node (11) at (1.2,1.2){$11$};
\path[thick,->,draw,black]
(00) edge[bend left=12] (01)
(01) edge[bend left=12] (00)
(10) edge[bend left=12] (11)
(11) edge[bend left=12] (10)
;
\end{tikzpicture}
\end{array}
&&
\begin{array}{l}
g^2_1(x)=x_1\\[2mm]
g^2_2(x)=x_2+1\\
\end{array}
&&
\begin{array}{c}
G(g^2)\\[2mm]
\begin{tikzpicture}
\useasboundingbox (-0.7,-0.4) rectangle (2.2,0.4);
\node[outer sep=1,inner sep=2,circle,draw,thick] (1) at (0,0){$1$};
\node[outer sep=1,inner sep=2,circle,draw,thick] (2) at (1.5,0){$2$};
\draw[->,thick] (1.{180+20}) .. controls (-0.8,-0.7) and (-0.8,+0.7) .. (1.{180-20});
\draw[->,thick] (2.{0+20}) .. controls ({1.5+0.8},+0.7) and ({1.5+0.8},-0.7) .. (2.{0-20});
\end{tikzpicture}
\end{array}
\end{array}
\\~\\
%g3
\begin{array}{ccccccc}
\begin{array}{c|c}
x & g^3(x)\\\hline
00&11\\
01&10\\
10&01\\
11&00
\end{array}
&&
\begin{array}{c}
\begin{tikzpicture}
\useasboundingbox (-0.3,-0.6) rectangle (1.5,1.8);
\node (00) at (0,0){$00$};
\node (01) at (0,1.2){$01$};
\node (10) at (1.2,0){$10$};
\node (11) at (1.2,1.2){$11$};
\path[thick,->,draw,black]
(00) edge[bend left=10] (11)
(01) edge[bend left=10] (10)
(10) edge[bend left=10] (01)
(11) edge[bend left=10] (00)
;
\end{tikzpicture}
\end{array}
&&
\begin{array}{l}
g^3_1(x)=x_1+1\\[2mm]
g^3_2(x)=x_2+1\\
\end{array}
&&
\begin{array}{c}
G(g^3)\\[2mm]
\begin{tikzpicture}
\useasboundingbox (-0.7,-0.4) rectangle (2.2,0.4);
\node[outer sep=1,inner sep=2,circle,draw,thick] (1) at (0,0){$1$};
\node[outer sep=1,inner sep=2,circle,draw,thick] (2) at (1.5,0){$2$};
\draw[->,thick] (1.{180+20}) .. controls (-0.8,-0.7) and (-0.8,+0.7) .. (1.{180-20});
\draw[->,thick] (2.{0+20}) .. controls ({1.5+0.8},+0.7) and ({1.5+0.8},-0.7) .. (2.{0-20});
\end{tikzpicture}
\end{array}
\end{array}
\end{array}
\]

\section{Lemma for the proof of Theorem~\ref{thm:main1}}

\begin{lemma}\label{lem:X}
Let $1\leq k\leq n$ and let $s_1,\dots,s_{2k}$ be positive integers with sum equal to $n+k$. If $n\geq 5$ there are disjoint subsets $X_1,\dots,X_{2k}\subseteq \B^n$ of size $s_1,\dots,s_{2k}$ such that, for every $i\in [n]$, there is $p\in [k]$ and $x\in X_{2p-1}$ with $x+e_i\in X_{2p}$.  
\end{lemma}

\begin{proof}
Let $I_1,\dots,I_{2k}$ be a partition of $[n]$ (with some members possibly empty) such that, for $1\leq \ell\leq 2k$, the size of $I_\ell$ is $s_\ell-1$ if $\ell$ is odd, and $s_\ell$ otherwise; it exists since the sum of the $s_\ell$ is $n+k$. For $1\leq p\leq k$, select a configuration $x^{2p-1}\in\B^n$, a component $j_{2p}\in I_{2p}$, and let 
\[
\begin{array}{lll}
X_{2p-1}&=&\{x^{2p-1}\}\cup\{x^{2p-1}+e_{j_{2p}}+e_i\mid i\in I_{2p-1}\},\\[1mm]
X_{2p}&=&\{x^{2p-1}+e_i\mid i\in I_{2p}\}. 
\end{array}
\]
Clearly, $|X_{2p-1}|=|I_{2p-1}|+1=s_{2p-1}$ and $|X_{2p}|=|I_{2p}|=s_{2p}$. Let $i\in [n]$. Then $i\in I_\ell$ for some $1\leq \ell\leq 2k$. If $\ell=2p-1$ then, setting $x=x^{2p-1}+e_{j_{2p}}+e_i$, we have $x\in X_{2p-1}$ and $x+e_i=x^{2p-1}+e_{j_{2p}}\in X_{2p}$. If $\ell=2p$ then, setting $x=x^{2p-1}$, we have $x\in X_{2p-1}$ and $x+e_i\in X_{2p}$. Thus we only have to prove that we can choose the configurations $x^{2p-1}$ so that the sets $X_1,\dots,X_{2k}$ are pairwise disjoint. This is obvious if $k=1$ (and it works for any choice of $x^1$). If $k=2$, then, since $n\geq 5$, one easily check that $X_1,\dots,X_{2k}$ are disjoint by taking $x^1=\ZERO$ and $x^3=\ONE$. Suppose now that $k\geq 3$ and choose $x^{2p-1}=e_{j_{2p-2}}$ for all $p\in[k]$, where $j_0$ means~$j_{2k}$. Then each $X_{2p-1}$ contains configurations of weight $1$ or $3$ and each $X_{2p}$ contains configurations of weight $2$. Hence, given $1\leq p<q\leq k$, we have to prove that $X_{2p-1}\cap X_{2q-1}=\emptyset$ and $X_{2p}\cap X_{2q}=\emptyset$. Suppose that $x\in X_{2p-1}\cap X_{2q-1}$. If $w(x)=1$ then we deduce that $x=e_{j_{2p-2}}=e_{j_{2q-2}}$ which is false since $p\neq q$. If $w(x)=3$ then $x_i=1$ for some $i\in I_{2p-1}$ while $y_i=0$ for all $y\in X_{2q-1}$, a contradiction. Thus $X_{2p-1}\cap X_{2q-1}=\emptyset$. Suppose now that $x\in X_{2p}\cap X_{2q}$. Then $x=e_{j_{2p-2}}+e_{i_{2p}}=e_{j_{2q-2}}+e_{i_{2q}}$ for some $i_{2p}\in I_{2p}$ and $i_{2q}\in I_{2q}$. Thus $i_{2p}=j_{2q-2}$ and $i_{2q}=j_{2p-2}$. Hence $j_{2q-2}\in I_{2p}$, and since $p<q$ this implies $q=p+1$. Also, $j_{2p-1}\in I_{2q}$ and since $p<q$ this implies $p=1$ and $q=k$, but then $q\neq p+1$ since $k\geq 3$. Thus indeed $X_{2p}\cap X_{2q}=\emptyset$. Hence the sets $X_1,\dots,X_{2k}$ are indeed pairwise disjoint.
\end{proof}

\section{Lemma for the proof of Theorem~\ref{thm:main2}}

We prove Lemma~\ref{lem:nice}, restated below. We first give some definitions. Given $X\subseteq \B^n$ and $i\in [n]$, we say that $X$ is {\em closed} by $e_i$ if $X=X+e_i$. One easily check that if $X,Y\subseteq\B^n$ are closed by $e_i$ so is $X\cap Y$. Also, if $|X|$ is closed by $e_i$ then $|X|$ is even. Indeed, let $X_0$ be the set of of $x\in X$ with $x_i=0$ and $X_1=X\setminus X_0$. We have $X_0+e_i\subseteq X_1$, thus $|X_0|\leq |X_1|$ and similarly, $X_1+e_i\subseteq X_0$ thus $|X_1|\leq |X_0|$. Hence $|X_0|=|X_1|$ so $|X|$ is even. 

\setcounter{lemma}{3}
\begin{lemma}
Let $f$ be a $n$-component network and distinct $i,j\in [n]$. Some digraph in $\G(f)$ has no arc from $j$ to $i$ if and only if $f$ has a $(2^{n-2})$-nice~set. 
\end{lemma}

\begin{proof}
Let $h=\pi\circ f\circ\pi^{-1}$ for some permutation $\pi$ of $\B^n$, and suppose that $G(h)$ has no arc from $j$ to $i$. Let $X$ be the set of $x\in\B^n$ with $x_i=0$, of size $2^{n-1}$, and let $X^-=h^{-1}(X)$. If $x\in X^-$, that is, $h_i(x)=0$, then $h_i(x+e_j)=0$ since otherwise $G(h)$ has an arc from $j$ to $i$. Thus $x+e_j\in X^-$. Hence $X^-$ is closed by $e_j$ thus $|X^-|$ is even. Since $j\neq i$, $X$ is also closed by~$e_j$, so $X^-\cap X$ is closed by $e_j$ and thus $|X^-\cap X|$ is even. Let $A=\pi^{-1}(X)$. Then one easily check that $f^{-1}(A)=\pi^{-1}(X^-)$ and that $f^{-1}(A)\cap A=\pi^{-1}(X^-\cap X)$. So $|f^{-1}(A)|$ and $|f^{-1}(A)\cap A|$ are even, and thus $A$ is a $(2^{n-2})$-nice set of $f$. 

Conversely, suppose that $f$ has a $(2^{n-2})$-nice set $A$, thus of size $2^{n-1}$. Let $A^-=f^{-1}(A)$. Then $|A|$, $|A^-|$ and $|A^-\cap A|$ are even, so $|A\setminus A^-|$ and $|A^-\setminus A|$ are also even. Hence there is a a balanced partition $(A_1,A_2)$ of $A\cap A^-$, a balanced partition $(A_3,A_4)$ of $A\setminus A^-$, and a balanced partition $(A^-_3,A^-_4)$ of $A^-\setminus A$. Then $(A_1,A_2,A_3,A_4)$ is a partition of $A$ and $(A_1,A_2,A^-_3,A^-_4)$ is a partition of $A^-$, and $A_1,A_2,A_3,A_4,A^-_3,A^-_4$ are disjoint. 

For $k=0,1$, let $Y_{k}$ be the set of $x\in \B^n$ with $x_i=k$ and $x_j=0$. Since $|A|=2^{n-1}$ we have $|A_1|+|A_3|=2^{n-2}$ thus there is a partition $(X_1,X_3)$ of $Y_0$ with $|X_1|=|A_1|$ and $|X_3|=|A_3|$. Since $|A^-_3|=|A^-\setminus A|/2\leq (2^n-|A|)/2=2^{n-2}$ there is $X^-_3\subseteq Y_1$ of size $|A^-_3|$. Let $X_2=X_1+e_j$, $X_4=X_3+e_j$ and $X^-_4=X^-_3+e_j$. Then $X_1,X_2,X_3,X_4,X^-_3,X^-_4$ are disjoint and of size $|A_1|,|A_2|,|A_3|,|A_4|,|A^-_3|,|A^-_4|$. Hence there is a permutation $\pi$ of $\B^n$ that sends $A_1,A_2,A_3,A_4,A^-_3,A^-_4$ on $X_1,X_2,X_3,X_4,X^-_3,X^-_4$, respectively. 

Let $h=\pi\circ f\circ\pi^{-1}$. Let $X=X_1\cup X_2\cup X_3\cup X_4$ and $X^-=X_1\cup X_2\cup X^-_3\cup X^-_4$. One easily check that: $X$ is the set of $x\in\B^n$ with $x_i=0$; $X^-$ is closed by $e_j$; and $X^-=h^{-1}(X)$. We will prove that $G(h)$ has no arc from $j$ to $i$. Suppose, for a contradiction, that there is $x\in\B^n$ with $h_i(x)\neq h_i(x+e_j)$. Without loss, we can assume that $h_i(x)=0$, that is, $h(x)\in X$. So $x\in X^-$, hence $x+e_j\in X^-$, thus $h(x+e_j)\in X$, that is, $h_i(x+e_j)=0$, a contradiction. Thus $G(h)$ has no arc from $j$ to $i$.
\end{proof}

\section{Lemma for the proof of Theorem~\ref{thm:main3}}

For $X\subseteq \B^n$, let $\partial(X)$ be the number of edges in $Q_n$ with exactly one end in $X$; equivalently, it is the number couples $(x,i)$ with $x\in X$ and $i\in [n]$ such that $x+e_i\not\in X$. We can regard $|X|$ as the volume of $X$ and $\delta(X)$ as its perimeter and ask: what is the minimum perimeter for a given volume? The answer is Harp's isoperimetric inequality, from 1964. We denote by $L_k$ the first $k$ configurations in $\B^n$ according to the lexicographic order (e.g. for $n=4$ we have $L_5=\{0000,1000,0100,1100,0010\}$). 

\begin{theorem}[Harper's isoperimetric inequality, 1964]
If $n\geq 1$ and~$X\subseteq \B^n$ then $\partial (X)\geq \partial(L_{|X|})$.  
\end{theorem}

In practice, the following approximation is often used; see \cite{KP20} for instance. 

\begin{corollary}\label{cor:harper}
If $n\geq 1$ and $X\subseteq \B^n$ is non-empty, then
\[
\partial (X)\geq |X|(n-\log_2|X|).
\]  
\end{corollary}

From this inequality, we deduce Lemma~\ref{lem:cube}, that we restate. 

\setcounter{lemma}{6}
\begin{lemma}
Let $A\subseteq \B^n$ be non-empty, and let $d$ be the average degree of the subgraph of $Q_n$ induced by $A$. Then $|A|\geq 2^d$.
\end{lemma}

\begin{proof}
For $a\in A$, let $d(a)$ be the degree of $a$ in the subgraph of $Q_n$ induced by $A$, and let $d$ be the average degree, that is, $d=\sum_{a\in A} d(a)/|A|$. Since, for each $a\in A$, there are exactly $n-d(a)$ components $i\in [n]$ such that $a+e_i\not\in A$, we have, using Corollary \ref{cor:harper}, 
\[
\partial (A)=\sum_{a\in A}n-d(a)=|A|n-|A|d=|A|(n-d)\geq |A|(n-\log_2|A|),
\]
thus $d\leq \log_2|A|$ as desired. 
\end{proof}

%%%%%%%%%%%%%%%%%%%%%%%%%%%%%%%%%%%
\end{document}